\documentclass[11pt]{article}
    
\usepackage{amsthm}
\usepackage{amssymb}
\usepackage{hyperref}                           
\usepackage{mathrsfs}
\usepackage{color} 
\usepackage{amsmath}                    
\usepackage{diagbox}

\def\sttf2#1#2{\left[\!\!\left[#1\atop#2\right]\!\!\right]}
\def\stf3f#1#2{\left[\!\!\left[\!\!\left[#1\atop#2\right]\!\!\right]\!\!\right]}
\def\stff4#1#2{\left[\!\!\left[\!\!\left[\!\!\left[#1\atop#2\right]\!\!\right]\!\!\right]\!\!\right]}
\def\stss2#1#2{\left\{\!\!\left\{#1\atop#2\right\}\!\!\right\}}

\newtheorem{theorem}{Theorem}

\newtheorem{Lem}{Lemma}

\allowdisplaybreaks

\begin{document}

\title{Fibonomial determinants}  

\author{ 
Takao Komatsu
\\ 
\small Institute of Mathematics\\[-0.8ex]
\small Henan Academy of Sciences\\[-0.8ex]
\small Zhengzhou 450046 China\\[-0.8ex]
\small \texttt{komatsu@zstu.edu.cn}\\
\small and\\[-0.8ex]
\small Department of Mathematics\\[-0.8ex] 
\small Institute of Science Tokyo\\[-0.8ex] 
\small 2-12-1 Ookayama, Meguro-ku\\[-0.8ex] 
\small Tokyo 152-8551 Japan\\[-0.8ex]
\small \texttt{komatsu.t.al@m.titech.ac.jp} 
}

\date{
}

\maketitle

\begin{abstract}  
In this paper, we find several determinants expressing the Fibonomial coefficients. We also give the generating functions, Vandermonde identity, and continued fractions about Fibonomial coefficients.   


\noindent 
{\bf MR Subject Classifications:} 11B39, 15A15.
\end{abstract}

\section{Introduction}

Fibonomial coefficients $\binom{n}{m}_F$ are defined by 
\begin{equation}  
\binom{n}{m}_F=\frac{F_n F_{n-1}\cdots F_1}{F_m F_{m-1}\cdots F_1\cdot F_{n-m}F_{n-m-1}\cdots F_1}\,, 
\label{def:fibonom}
\end{equation} 
where Fibonacci numbers are determined by 
\begin{equation}  
F_n=F_{n-1}+F_{n-2}\quad(n\ge 2),\quad F_0=0,\quad F_1=1\,. 
\label{def:fib}
\end{equation} 
The Fibonomial coefficients (\cite[A010048]{oeis}) are similar to binomial coefficients $\binom{n}{m}$ and can be displayed in a triangle similar to Pascal's triangle. From the definition (\ref{def:fibonom}), it is clear that 
\begin{equation}  
\binom{n}{m}_F=\binom{n}{n-m}_F\,. 
\label{symmetric:fibonom}
\end{equation} 
Fibonomial coefficients satisfy the recurrence relation 
\begin{equation}  
\binom{n}{m}_F=F_{n-m-1}\binom{n-1}{m-1}_F+F_{m+1}\binom{n-1}{m}_F
\label{recf:fibonom}
\end{equation} 
(see, e.g., \cite{Gould,Holte}).  
Several identities are derived in \cite{ST05,Trojovsky07}, including
$$
\sum_{j=0}^m(-1)^{\frac{j(m+j)}{2}}\binom{m}{j}_F=0\,.
$$   

Although research on determinants related to Fibonacci numbers in (\ref{def:fib}) has long been well known (see, e.g., \cite{Koshy1,Koshy2}), determinants that directly represent Fibonacci numbers themselves were hardly known (see, e.g., \cite{CDNN,Proskuryakov,Strang}). However, by using Cameron's operator (\cite{Cameron89}), determinants representing Fibonacci numbers have been provided (\cite{Ko20b}) as well as other numbers (\cite{KKV22}), and very recently, extensions to the determinants of Fibonacci polynomials (\cite{Ko25}) and further determinant representations of general Fibonacci polynomials (\cite{KK26}) have been successfully obtained. 

In \cite{Kilic10}, one of the determinant expressions of Fibonomials is given by 
the $n\times n$ determinant:  
\begin{equation}
\binom{n+3}{3}_F=\left|\begin{array}{cccccc}  
3&-1&0&&&0\\ 
6&3&-1&&&\\
-3&6&3&&&\\
-1&&&\ddots&&0\\
0&&&6&3&-1\\
&0&-1&-3&6&3 
\end{array}\right|\,.
\label{eq:emrah10}
\end{equation}
See also \cite{AK25}.  

On the other hand, several determinant expressions of binomial coefficients are known.  For example, one determinant representation in the form of a Hessenberg matrix is: 

\begin{equation}  
\binom{n+1}{k}=\left|\begin{array}{ccccc}
\binom{n+1}{n}&1&0&&0\\
\binom{n+2}{n}&\binom{n+1}{n}&1&&\\
\binom{n+3}{n}&\binom{n+2}{n}&\binom{n+1}{n}&&0\\
\vdots&&\ddots&\ddots&1\\
\binom{n+k}{n}&\cdots&\binom{n+3}{n}&\binom{n+2}{n}&\binom{n+1}{n}
\end{array}\right|\,,
\label{det:binom1}
\end{equation}

In this paper, we give some determinant representations of the Fibonomial coefficients in the form of a Hessenberg matrix (see, e.g., \cite{KT10,KTH10}), as an analogy to the determinant representations of the binomial coefficients. In particular, the expression of (\ref{eq:emrah10}) can be generalized.   
Furthermore, a Fibonomial version of the famous Vandermonde identity is given.  
As an application, the continued fraction expansions of the generating functions of the Fibonomial coefficients are shown. Several identities of the Fibonomial coefficients are also represented by using Trudi's formula.

\section{Main results}  

As an analogous result of (\ref{det:binom1}), we get the following.  

\begin{theorem}  
\begin{equation}  
\binom{n+1}{k}_F=(-1)^{\binom{k}{2}}\left|\begin{array}{ccccc}
\binom{n+1}{n}_F&1&0&&0\\
\binom{n+2}{n}_F&\binom{n+1}{n}_F&1&&\\
\binom{n+3}{n}_F&\binom{n+2}{n}_F&\binom{n+1}{n}_F&&0\\
\vdots&&\ddots&\ddots&1\\
\binom{n+k}{n}_F&\cdots&\binom{n+3}{n}_F&\binom{n+2}{n}_F&\binom{n+1}{n}_F
\end{array}\right|\,,
\label{det:fibonom1}
\end{equation}
where 
$$
(-1)^{\binom{k}{2}}=\begin{cases}
+1&\text{if $k\equiv 2,3\pmod 4$};\\
-1&\text{if $k\equiv 0,1\pmod 4$}. 
\end{cases}
$$
\label{th1}
\end{theorem}

In order to prove (\ref{det:binom1}), we need the following relation.  Then the actual proof is similarly done to that of Theorem \ref{th1}.  

\begin{Lem}
$$
\sum_{l=0}^k(-1)^{l}\binom{n+l}{n}\binom{n+1}{k-l}=0\,.
$$
\label{lem1}
\end{Lem}  
\begin{proof} 
By using the Chu-Vandermonde identity, the left-hand side can be written as 
$$  
\sum_{l=0}^k\binom{-n-1}{l}\binom{n+1}{k-l}=\binom{0}{k}=0\,,
$$
which is the right-hand side.  
\end{proof}

Similarly to Lemma \ref{lem1}, we need the followin relation as Fibonomial version.  

\begin{Lem}
For $\delta_k=(-1)^{\binom{k}{2}}$, 
$$
\sum_{l=0}^k(-1)^{l}\binom{n+l}{n}_F\delta_{k-l}\binom{n+1}{k-l}_F=0\,.
$$
\label{lem2}
\end{Lem}  
\begin{proof}
Consider the formal power series 
\begin{align*}  
A(z)&=\sum_{l=0}^\infty(-1)^{l}\binom{n+l}{n}_F z^l\,,\\ 
B(z)&=\sum_{m=0}^{n+1}\delta_m\binom{n+1}{m}_F z^m\,. 
\end{align*}
Note that Fibonacci numbers $F_m$ are given by 
$$
F_m=\frac{\alpha^m-\beta^m}{\alpha-\beta}\,,
$$ 
where 
$$
\alpha=\frac{1+\sqrt{5}}{2}\quad\hbox{and}\quad \beta=\frac{1-\sqrt{5}}{2}\,. 
$$ 
By using the Gaussian binomial coefficient 
$$
\binom{N}{K}_q=\prod_{r=1}^K\frac{1-q^{N-r+1}}{1-q^r}\,,
$$ 
we get for $q=\beta/\alpha$ 
\begin{align}
\binom{N}{K}_F&=\prod_{r=1}^K\frac{F_{N-r+1}}{F_r}=\prod_{r=1}^K\frac{\alpha^{N-r+1}-\beta^{N-r+1}}{\alpha^r-\beta^r}\notag\\
&=\prod_{r=1}^K\frac{\alpha^{N-r+1}(1-q^{N-r+1})}{\alpha^r(1-q^r)}\notag\\
&=\alpha^{K(N-K)}\binom{N}{K}_q\,. 
\label{eq:65}
\end{align}
Hence, by (\ref{eq:65}) 
$$
A(z)=\sum_{l=0}^\infty(-1)^l\alpha^{n l}\binom{n+l}{n}_q z^l
$$ 
Since a $q$-binomial theorem is 
$$
\prod_{j=0}^{N-1}\frac{1}{1-x q^j}=\sum_{k=0}^\infty\binom{N+k-1}{k}_q x^k\,, 
$$
substituting $x=-\alpha^n z$ and $N=n+1$, we have 
$$
A(z)=\prod_{j=0}^n\frac{1}{1+\alpha^n q^j z}=\prod_{j=0}^n\frac{1}{1+\alpha^{n-j}\beta^j z}\,.
$$

By applying (\ref{eq:65}) and $\alpha\beta=-1$, 
\begin{align*}
B(z)&=\sum_{m=0}^{n+1}(-1)^{\binom{m}{2}}\alpha^{m(n+1-m)}\binom{n+1}{m}_q z^m\\
&=\sum_{m=0}^{n+1}(\alpha\beta)^{\binom{m}{2}}\alpha^{m(n+1-m)}\binom{n+1}{m}_q z^m\\
&=\sum_{m=0}^{n+1}q^{\binom{m}{2}}\alpha^{m n}\binom{n+1}{m}_q z^m\,.
\end{align*}
Since another $q$-binomial theorem is 
$$
\prod_{j=0}^{N-1}(1+x q^j)=\sum_{m=0}^N q^{\frac{m(m-1)}{2}}\binom{N}{m}_q x^m\,,
$$
substituting $x=\alpha^n z$ and $N=n+1$, we have 
$$
B(z)=\prod_{j=0}^n(1+\alpha^n q^j z)=\prod_{j=0}^n(1+\alpha^{n-j}\beta^j z)\,.
$$

Therefore,  
\begin{align*}  
1&=A(z)B(z)\\
&=\sum_{k=0}^\infty\sum_{l=0}^k(-1)^{l}\binom{n+l}{n}_F(-1)^{\binom{k-l}{2}}\binom{n+1}{k-l}_F z^k\,. 
\end{align*} 
Comparing the coefficients on both sides, for $k\ge 1$, 
$$
\sum_{l=0}^k(-1)^{l}\binom{n+l}{n}_F(-1)^{\binom{k-l}{2}}\binom{n+1}{k-l}_F=0\,. 
$$ 
\end{proof}

\begin{proof}[Proof of Theorem \ref{th1}.]
When $k=1$, 
$$
{\rm RHS}=\binom{n+1}{n}_F=\binom{n+1}{1}_F={\rm LHS}\,.
$$ 
Assume that the identity (\ref{det:fibonom1}) holds for $k=1,2,\dots,\kappa-1$. 
Repeatedly expanding the determinant of the right-hand side of equation (\ref{det:fibonom1}) when $k=\kappa$ along the first row gives 
\begin{align*}
&\binom{n+1}{n}_F\delta_{\kappa-1}\binom{n+1}{\kappa-1}_F-\left|\begin{array}{ccccc}
\binom{n+2}{n}_F&1&0&&0\\
\binom{n+3}{n}_F&\binom{n+1}{n}_F&1&&\\
\binom{n+4}{n}_F&\binom{n+2}{n}_F&\binom{n+1}{n}_F&&0\\
\vdots&&\ddots&\ddots&1\\
\binom{n+\kappa}{n}_F&\cdots&\binom{n+3}{n}_F&\binom{n+2}{n}_F&\binom{n+1}{n}_F
\end{array}\right|\\
&=\binom{n+1}{n}_F\delta_{\kappa-1}\binom{n+1}{\kappa-1}_F-\binom{n+2}{n}_F\delta_{\kappa-2}\binom{n+1}{\kappa-2}_F\\
&\quad +\left|\begin{array}{ccccc}
\binom{n+3}{n}_F&1&0&&0\\
\binom{n+4}{n}_F&\binom{n+1}{n}_F&1&&\\
\binom{n+5}{n}_F&\binom{n+2}{n}_F&\binom{n+1}{n}_F&&0\\
\vdots&&\ddots&\ddots&1\\
\binom{n+\kappa}{n}_F&\cdots&\binom{n+3}{n}_F&\binom{n+2}{n}_F&\binom{n+1}{n}_F
\end{array}\right|\\
&=\binom{n+1}{n}_F\delta_{\kappa-1}\binom{n+1}{\kappa-1}_F-\binom{n+2}{n}_F\delta_{\kappa-2}\binom{n+1}{\kappa-2}_F\\
&\quad +\binom{n+3}{n}_F\delta_{\kappa-3}\binom{n+1}{\kappa-3}_F-\cdots+(-1)^\kappa\left|\begin{array}{cc}  
\binom{n+\kappa-1}{n}_F&1\\
\binom{n+\kappa}{n}_F&\binom{n+1}{n}_F
\end{array}\right|\\
&=\sum_{l=1}^{\kappa-1}(-1)^{l-1}\binom{n+l}{n}_F\delta_{\kappa-l}\binom{n+1}{\kappa-l}_F=\delta_\kappa\binom{n+1}{\kappa}_F\,. 
\end{align*}
We used Lemma \ref{lem2} to derive the final equality.  
As the identity (\ref{det:fibonom1}) holds for $k=\kappa$, by induction, the theorem is proved.   
\end{proof}

\subsection{Fibonomial version of the Vandermonde identity}

Related to Lemma \ref{lem2}, we shall give a Fibonomial analogue of the famous Chu-Vandermonde identity
$$
\binom{m+n}{k}=\sum_{j=0}^k\binom{m}{k-j}\binom{n}{j}\,.
$$

\begin{theorem}  
\begin{align*}
\binom{m+n}{k}_F&=\sum_{j=0}^k\binom{m}{k-j}_F\binom{n}{j}_F\alpha^{(n-j)(k-j)}\beta^{(m-k+j)j}\\
&=\sum_{j=0}^k\binom{m}{k-j}_F\binom{n}{j}_F\alpha^{n(k-j)}\beta^{m j}(-1)^{j(k-j)}\,,
\end{align*}
where 
$$
\alpha=\frac{1+\sqrt{5}}{2}\quad\hbox{and}\quad \beta=\frac{1-\sqrt{5}}{2}\,. 
$$ 
\label{th:fibonom-vandermonde}
\end{theorem}  
\begin{proof}
The identity can be obtained directly by applying (\ref{eq:65}) to the $q$-Vandermonde identity: 
$$
\binom{m+n}{k}_q=\sum_{j=0}^k\binom{m}{k-j}_q\binom{n}{j}_q q^{j(m-k+j)}
$$ 
(see, e.g., \cite[p.190, No.100]{Stanley1}). 
\end{proof}

Nevertheless, here is another proof by using the generating function.  

\begin{proof}[Another proof of Theorem \ref{th:fibonom-vandermonde}.]  
Using Lemma \ref{gf:fibonom} below, 
\begin{align*}
&\sum_{k=0}^n\binom{m+n}{k}_F(-1)^{\binom{k}{2}}z^k\\
&=\prod_{j=0}^{m+n-1}(1+\alpha^{m+n-1-j}\beta^j z)\\
&=\prod_{j=0}^{m}(1+\alpha^{m+n-1-j}\beta^j z)\prod_{j=m}^{m+n-1}(1+\alpha^{m+n-1-j}\beta^j z)\\
&=\prod_{j=0}^{m}(1+\alpha^{m+n-1-j}\beta^j z)\prod_{\ell=0}^{n}(1+\alpha^{n-1-\ell}\beta^{m+\ell}z)\\
&=\left(\sum_{r=0}^m\binom{m}{r}_F(-1)^{\binom{r}{2}}(\alpha^n z)^r\right)\left(\sum_{l=0}^n\binom{n}{l}_F(-1)^{\binom{l}{2}}(\beta^m z)^l\right)\\
&=\sum_{r+l=k}\binom{m}{r}_F\binom{n}{l}_F(-1)^{\binom{r}{2}+\binom{l}{2}}\alpha^{n r}\beta^{m l}z^{r+l}\\
&=\sum_{j=0}^k\binom{m}{k-j}_F\binom{n}{j}_F(-1)^{\binom{k-j}{2}+\binom{j}{2}}\alpha^{n(k-j)}\beta^{m j}z^{k}\,. 
\end{align*} 
Since 
$$
\binom{k-j}{2}+\binom{j}{2}-\binom{k}{2}=j(j-k)\,, 
$$ 
comparing the coefficients on both sides, we get the desired result.  
\end{proof}

\begin{Lem}  
$$
\sum_{k=0}^n\binom{n}{k}_F(-1)^{\binom{k}{2}}z^k=\prod_{j=0}^{n-1}(1+\alpha^{n-1-j}\beta^j z)\,. 
$$
\label{gf:fibonom}
\end{Lem}
\begin{proof}
The right-hand side can be written as  
\begin{align*}
\prod_{j=0}^{n-1}(1+\alpha^{n-1-j}\beta^j z)
&=\sum_{k=0}^n\left(\sum_{0\le i_1<\dots<i_k\le n-1}\prod_{t=1}^k\alpha^{n-1-i_t}\beta^{i_t}\right)z^k\\
&=\sum_{k=0}^{n}\left(\alpha^{k(n-1)}\sum_{J}\left(\frac{\beta}{\alpha}\right)^{\sum_{i\in J}i}\right)z^k\,,
\end{align*}
where the sum runs over all $k$-subsets $J\subset\{0,1,\dots,n-1\}$.  
Since 
$$
\sum_{0\le i_1<\dots<i_k\le n-1}q^{i_1+\cdots+i_k}=q^{\binom{k}{2}}\prod_{i=1}^k\frac{1-q^{n-k+i}}{1-q^i}\,, 
$$ 
for $q=\beta/\alpha=-1/\alpha^2$, the right-hand side is equal to 
\begin{align*}
&\sum_{k=0}^{n}\left(\alpha^{k(n-1)}q^{\binom{k}{2}}\prod_{i=1}^k\frac{1-q^{n-k+i}}{1-q^i}\right)z^k\\
&=\sum_{k=0}^{n}\left(\alpha^{k(n-1)}\left(-\frac{1}{\alpha^2}\right)^{\binom{k}{2}}\prod_{i=1}^k\frac{1-(\beta/\alpha)^{n-k+i}}{1-(\beta/\alpha)^i}\right)z^k\\
&=\sum_{k=0}^{n}\left((-1)^{\binom{k}{2}}\prod_{i=1}^k\frac{\alpha^{n-k+i}-\beta^{n-k+i}}{\alpha^i-\beta^i}\right)z^k\\
&=\sum_{k=0}^{n}\binom{n}{k}_F(-1)^{\binom{k}{2}}z^k\,.
\end{align*} 
Here, we used an expression in (\ref{eq:65}). 
\end{proof}

\section{Applications of the inversion formula}

By using the inversion formula explained as follows (see, e.g., \cite{KR18}), we can get an expression in which the inner and outer elements are interchanged.  

\begin{Lem}  
For two sequences $\{\alpha_n\}_{n\ge 0}$ and $\{\beta_n\}_{n\ge 0}$ with $\alpha_0=\beta_0=1$, we have for $n\ge 1$
\begin{align*}
&\alpha_n=\left|\begin{array}{ccccc}
\beta_1&1&0&\cdots&0\\
\beta_2&\beta_1&1&&\vdots\\
\vdots&&\ddots&&0\\
\beta_{n-1}&&&\beta_1&1\\
\beta_n&\beta_{n-1}&\cdots&\beta_2&\beta_1
\end{array}\right|\\
&\Longleftrightarrow\quad 
\sum_{k=0}^n(-1)^{n-k}\alpha_k\beta_{n-k}=0\\
&\Longleftrightarrow\quad 
\beta_n=\left|\begin{array}{ccccc}
\alpha_1&1&0&\cdots&0\\
\alpha_2&\alpha_1&1&&\vdots\\
\vdots&&\ddots&&0\\
\alpha_{n-1}&&&\alpha_1&1\\
\alpha_n&\alpha_{n-1}&\cdots&\alpha_2&\alpha_1
\end{array}\right|\,. 
\end{align*}
\label{lem:inversion1} 
\end{Lem}

Applying Lemma \ref{lem:inversion1} to Theorem \ref{th1}, we obtain the following.

\begin{theorem}  
\begin{multline*}  
\binom{n+k}{k}_F\\
=\left|\begin{array}{ccccc}
\binom{n+1}{1}_F&1&0&&0\\
(-1)^{\binom{2}{2}}\binom{n+1}{2}_F&\binom{n+1}{1}_F&1&&\\
(-1)^{\binom{3}{2}}\binom{n+1}{3}_F&(-1)^{\binom{2}{2}}\binom{n+1}{2}_F&\binom{n+1}{1}_F&&0\\
\vdots&&\ddots&\ddots&1\\
(-1)^{\binom{k}{2}}\binom{n+1}{k}_F&\cdots&(-1)^{\binom{3}{2}}\binom{n+1}{3}_F&(-1)^{\binom{2}{2}}\binom{n+1}{2}_F&\binom{n+1}{1}_F
\end{array}\right|\,. 
\end{multline*}
\label{tha1} 
\end{theorem}

In fact, it is easy to see that we have a slightly generalized expression for real $\alpha(\ne 0)$:  

\begin{theorem}  
\begin{multline*}  
\binom{n+k}{k}_F\\
=\left|\begin{array}{ccccc}
\binom{n+1}{1}_F&\alpha&0&&0\\
\frac{(-1)^{\binom{2}{2}}}{\alpha}\binom{n+1}{2}_F&\binom{n+1}{1}_F&\alpha&&\\
\frac{(-1)^{\binom{3}{2}}}{\alpha^2}\binom{n+1}{3}_F&\frac{(-1)^{\binom{2}{2}}}{\alpha}\binom{n+1}{2}_F&\binom{n+1}{1}_F&&0\\
\vdots&&\ddots&\ddots&\alpha\\
\frac{(-1)^{\binom{k}{2}}}{\alpha^{k-1}}\binom{n+1}{k}_F&\cdots&\frac{(-1)^{\binom{3}{2}}}{\alpha^2}\binom{n+1}{3}_F&\frac{(-1)^{\binom{2}{2}}}{\alpha}\binom{n+1}{2}_F&\binom{n+1}{1}_F
\end{array}\right|\,. 
\end{multline*}
\label{tha1-1} 
\end{theorem}

\noindent 
{\it Remark.}  
When $\alpha=-1$ and $k=3$, the expression of (\ref{eq:emrah10}) is recovered.

\section{Applications by Trudi's formula}  

The forms of determinant expressions in Theorem \ref{th1} remind us Trudi's formula \cite[Vol. 3, p. 214]{Muir}, \cite{Trudi}, and the case $a_0=1$ of this formula is known as Brioschi's formula \cite[Vol. 3, pp. 208--209]{Muir}, cite{Brioschi}.  

\begin{Lem}  
For a positive integer n, we have 
\begin{align*}  
&\left|\begin{array}{ccccc} 
a_1&a_0&0&\cdots&0\\
a_2&a_1&\vdots&&\vdots\\
\vdots&\vdots&\ddots&\ddots&0\\
a_{n-1}&&\cdots&a_1&a_0\\
a_n&a_{n-1}&\cdots&a_2&a_1
\end{array}\right|\\
&=\sum_{t_1+2 t_2+\cdots+n t_n=n\atop t_1,t_2,\dots,t_n\ge 0}\binom{t_1+\cdots+t_n}{t_1,\dots,t_n}(-a_0)^{n-t_1-\cdots-t_n}a_1^{t_1}a_2^{t_2}\cdots a_n^{t_n}\,,
\end{align*} 
where $\displaystyle\binom{t_1+\cdots+t_n}{t_1,\dots,t_n}=\dfrac{(t_1+\cdots+t_n)!}{t_1!\cdots t_n!}$ are the multinomial coefficients.  
\label{lem:trudi} 
\end{Lem}

By applying Lemma \ref{lem:trudi} to Theorem \ref{th1}, we obtain the following. 
\begin{theorem} 
We have 
\begin{multline*}
\binom{n+1}{k}_F
=\sum_{t_1+2 t_2+\cdots+n t_k=k\atop t_1,t_2,\dots,t_k\ge 0}\binom{t_1+\cdots+t_k}{t_1,\dots,t_k}\\
\times(-1)^{\binom{k}{2}+n-t_1-\cdots-t_k}\binom{n+1}{n}_F^{t_1}\binom{n+2}{n}_F^{t_2}\cdots\binom{n+k}{n}_F^{t_k}\,. 
\end{multline*}
\label{th:5}
\end{theorem}

By applying Lemma \ref{lem:trudi} to Theorem \ref{tha1}, by 
\begin{align*}
&n-t_1-\cdots-t_k+\binom{1}{2}t_1+\cdot+\binom{k}{2}t_k\\
&\quad =n+\sum_{j=1}^k\frac{(j+1)(j-2)}{2}t_j\,, 
\end{align*}
we obtain the following.  

\begin{theorem} 
We have 
\begin{multline*}
\binom{n+k}{k}_F
=\sum_{t_1+2 t_2+\cdots+n t_k=k\atop t_1,t_2,\dots,t_k\ge 0}\binom{t_1+\cdots+t_k}{t_1,\dots,t_k}\\
\times(-1)^{n+\sum_{j=1}^k\frac{(j+1)(j-2)}{2}t_j}\binom{n+1}{1}_F^{t_1}\binom{n+1}{2}_F^{t_2}\cdots\binom{n+1}{k}_F^{t_k}\,. 
\end{multline*}
\label{th:a5}
\end{theorem}

\section{An application to continued fractions}  

The determinant representation in the form of a Hessenberg matrix is very useful and versatile. Here we show a direct application to the continued fraction representation of generating functions. See \cite{Ko20} for similar continued fractions.   

First, the continued fraction expansion corresponding to Theorem \ref{th1} is as follows:

\begin{theorem}  
We have 
\begin{align*}
&\sum_{k=0}^{n+1}\binom{n+1}{k}_F(-1)^{\binom{k+1}{2}}x^k\\
&\equiv 1-\cfrac{F_{n+1}x}{1+F_{n+1}x-\cfrac{\frac{F_{n+2}}{F_2}x}{1+\frac{F_{n+2}}{F_2}x-{\atop\ddots-\cfrac{\frac{F_{2 n+1}}{F_{n+1}}x}{1+\frac{F_{2 n+1}}{F_{n+1}}x}}}}\pmod{x^{n+2}}\,. 
\end{align*}
\label{th:cf:fibonom}  
\end{theorem}    
\begin{proof}  
Applying the fundamental Lemma to Theorem \ref{th1}, 
\begin{align*}
&\sum_{k=0}^{n+1}\binom{n+1}{k}_F(-1)^{\binom{k}{2}}(-x)^k\\
&\equiv 1-\cfrac{\binom{n+1}{n}_F x}{1+\binom{n+1}{n}_F x-\cfrac{\frac{\binom{n+2}{n}_F}{\binom{n+1}{n}_F}x}{1+\frac{\binom{n+2}{n}_F}{\binom{n+1}{n}_F}x-{\atop\ddots-\cfrac{\frac{\binom{2 n+1}{n}_F}{\binom{2 n}{n}_F}x}{1+\frac{\binom{2 n+1}{n}_F}{\binom{2 n}{n}_F}x}}}}\pmod{x^{n+2}}\,.
\end{align*}
Since 
$$
\frac{\binom{n+k}{n}_F}{\binom{n+k-1}{n}_F}=\frac{F_{n+k}}{F_k}\quad(k\ge 1)
$$ 
and $(-1)^{\binom{k}{2}}(-1)^k=(-1)^{\binom{k+1}{2}}$, we get the desired result.
\end{proof}

\noindent 
{\bf Example.}  
When $n=6$, 
$$
\sum_{k=0}^7\binom{7}{k}_F(-1)^{\binom{k+1}{2}}x^k=1-13 x-104 x^2+260 x^3+260 x^4-104 x^5-13 x^6+x^7\,.  
$$ 
On the other hand, the corresponding continued fraction yields 
{\small  
\begin{align*}  
&1-\cfrac{F_7 x}{1+F_7 x-\cfrac{\frac{F_8}{F_2}x}{1+\frac{F_8}{F_2}x-\cfrac{\frac{F_9}{F_3}x}{1+\frac{F_9}{F_3}x-\cfrac{\frac{F_{10}}{F_4}x}{1+\frac{F_{10}}{F_4}x-\cfrac{\frac{F_{11}}{F_5}x}{1+\frac{F_{11}}{F_5}x-\cfrac{\frac{F_{12}}{F_6}x}{1+\frac{F_{12}}{F_6}x-\cfrac{\frac{F_{13}}{F_7}x}{1+\frac{F_{13}}{F_7}x}}}}}}}\\
&=1-\cfrac{13 x}{1+13 x-\cfrac{\frac{21}{1}x}{1+\frac{21}{1}x-\cfrac{\frac{34}{2}x}{1+\frac{34}{2}x-\cfrac{\frac{55}{3}x}{1+\frac{55}{3}x-\cfrac{\frac{89}{5}x}{1+\frac{89}{5}x-\cfrac{\frac{134}{8}x}{1+\frac{134}{8}x-\cfrac{\frac{223}{13}x}{1+\frac{223}{13}x}}}}}}}\\
&=1-13 x-104 x^2+260 x^3+260 x^4-104 x^5-13 x^6+x^7+8771626578 x^8+\cdots\,. 
\end{align*}
}

Next, the continued fraction expansion corresponding to Theorem \ref{tha1} is as follows:

\begin{theorem}  
We have 
\begin{align*}
&\sum_{k=0}^{\infty}\binom{n+k}{k}_Fx^k\\
&\equiv 1+\cfrac{F_{n+1}x}{1-F_{n+1}x-\cfrac{\frac{F_{n}}{F_2}x}{1+\frac{F_{n}}{F_2}x-{\atop\ddots-(-1)^n\cfrac{\frac{F_{1}}{F_{n+1}}x}{1+(-1)^n\frac{F_{1}}{F_{n+1}}x}}}}\pmod{x^{n+2}}\,. 
\end{align*}
\label{th:cf:afibonom}  
\end{theorem}    
\begin{proof}  
Applying the fundamental Lemma to Theorem \ref{tha1}, 
\begin{align*}
&\sum_{k=0}^{\infty}\binom{n+k}{k}_F x^k\\
&\equiv 1-\cfrac{\delta_1\binom{n+1}{1}_F x}{1+\delta_1\binom{n+1}{1}_F x-\cfrac{\frac{\delta_2\binom{n+1}{2}_F}{\delta_1\binom{n+1}{1}_F}x}{1+\frac{\delta_2\binom{n+1}{2}_F}{\delta_1\binom{n+1}{1}_F}x-{\atop\ddots-\cfrac{\frac{\delta_{n+1}\binom{n+1}{n+1}_F}{\delta_n\binom{n+1}{n}_F}x}{1+\frac{\delta_{n+1}\binom{n+1}{n+1}_F}{\delta_n\binom{n+1}{n}_F}x}}}}\pmod{x^{n+2}}\,.
\end{align*}
Since 
$$
\frac{(-1)^{\binom{k+1}{2}}\binom{n+1}{k+1}_F}{(-1)^{\binom{k}{2}}\binom{n+1}{k}_F}=(-1)^{k^2}\frac{F_{n-k+1}}{F_{k+1}}=(-1)^k\frac{F_{n-k+1}}{F_{k+1}}\quad(1\le k\le n)\,,
$$ 
we get the desired result.
\end{proof}

\noindent 
{\bf Example.}  
When $n=7$, 
\begin{multline*}
\sum_{k=0}^\infty\binom{7+k}{k}_F x^k=1+21 x+714 x^2+19635 x^3+582505 x^4+16776144 x^5\\
+488605194 x^6+14169550626 x^7+411591708660 x^8+\cdots\,.  
\end{multline*}
On the other hand, the corresponding continued fraction yields 
{\small  
\begin{align*}  
&1+\cfrac{F_8 x}{1-F_8 x-\cfrac{\frac{F_7}{F_2}x}{1+\frac{F_7}{F_2}x+\cfrac{\frac{F_6}{F_3}x}{1-\frac{F_6}{F_3}x-\cfrac{\frac{F_5}{F_4}x}{1+\frac{F_5}{F_4}x+\cfrac{\frac{F_4}{F_5}x}{1-\frac{F_4}{F_5}x-\cfrac{\frac{F_3}{F_6}x}{1+\frac{F_3}{F_6}x+\cfrac{\frac{F_2}{F_7}x}{1-\frac{F_2}{F_7}x-\cfrac{\frac{F_1}{F_8}x}{1+\frac{F_1}{F_8}x}}}}}}}}\\
&=1+21 x+714 x^2+19635 x^3+582505 x^4+16776144 x^5\\
&\quad +488605194 x^6+14169550626 x^7+411591708660 x^8+\cdots\,. 
\end{align*} 
}

\subsection{Fundamental frame}   

By combining the methods between determinants and continued fractions in \cite{Ko20,Ko21}, we have the following equivalent relations (\cite[Theorem 1]{KKal26a}).  

\begin{Lem}  
For three sequences $\{f_n\}_{n\ge 0}$ (with $f_0=0$), $\{g_n\}_{n\ge 1}$ and $\{h_n\}_{n\ge 1}$, we have 
\begin{align*}  
&\sum_{n=0}^\infty f_n x^n=\left(\sum_{j=0}^\infty\frac{h_1\cdots h_j}{g_1\cdots g_j}x^j\right)^{-1}:=\left(\sum_{j=0}^\infty\frac{H_j}{G_j}x^j\right)^{-1}\\
&=1-\cfrac{h_1 x}{g_1+h_1 x-\cfrac{g_1 h_2 x}{g_2+h_2 x-\cfrac{g_2h_3 x}{g_3+h_3 x-\ddots}}}\\
&\Longleftrightarrow\\ 
&f_n=(-1)^n\left|\begin{array}{ccccc}
\frac{H_1}{G_1}&1&0&&\\  
\frac{H_2}{G_2}&\frac{H_1}{G_1}&1&&\\
\vdots&\vdots&\ddots&&0\\
\frac{H_{n-1}}{G_{n-1}}&\frac{H_{n-2}}{G_{n-2}}&\cdots&\frac{H_1}{G_1}&1\\ 
\frac{H_n}{G_n}&\frac{H_{n-1}}{G_{n-1}}&\cdots&\frac{H_2}{G_2}&\frac{H_1}{G_1}
\end{array}\right|\\
&\Longleftrightarrow\quad\sum_{k=0}^n\frac{f_k H_{n-k}}{G_{n-k}}=\begin{cases}
1&\text{if $n=0$};\\
0&\text{if $n\ge 1$}
\end{cases}\\ 
&\Longleftrightarrow\\ 
&\frac{H_n}{G_n}=(-1)^n\left|\begin{array}{ccccc}
f_1&1&0&&\\  
f_2&f_1&1&&\\
\vdots&\vdots&\ddots&&0\\
f_{n-1}&f_{n-2}&\cdots&f_1&1\\ 
f_n&f_{n-1}&\cdots&f_2&f_1
\end{array}\right|\\
&\Longleftrightarrow\\ 
&f_n=\sum_{t_1+2 t_2+\cdots+n t_n=n}\binom{t_1+\cdots+t_n}{t_1,\dots,t_n}(-1)^{t_1+\cdots+t_n}\left(\frac{H_1}{G_1}\right)^{t_1}\left(\frac{H_2}{G_2}\right)^{t_2}\cdots\left(\frac{H_n}{G_n}\right)^{t_n}\\
&\phantom{f_n}=\sum_{k=1}^n(-1)^k\sum_{i_1+\cdots+i_k=n\atop i_1,\dots,i_k\ge 1}\frac{H_{i_1}}{G_{i_1}}\cdots\frac{H_{i_k}}{G_{i_k}}\\
&\Longleftrightarrow\\ 
&\frac{H_n}{G_n}=\sum_{t_1+2 t_2+\cdots+n t_n=n}\binom{t_1+\cdots+t_n}{t_1,\dots,t_n}(-1)^{t_1+\cdots+t_n}f_1^{t_1}f_2^{t_2}\cdots f_n^{t_n}\\
&\phantom{\frac{H_n}{G_n}}=\sum_{k=1}^n(-1)^k\sum_{i_1+\cdots+i_k=n\atop i_1,\dots,i_k\ge 1}f_{i_1}\cdots f_{i_k}\,.   
\end{align*}
\label{lem:det-cont}
\end{Lem}

\section{Bell polynomials}   

In Theorem \ref{th1}, we find a determinant expression of Fibonomial coefficients. In fact, with the help of Bell polynomials, we can still have a more determinant expression of a slightly modified form.  

The (complete exponential) Bell polynomial $\mathbf Y_n(x_1,x_2,\dots,x_n)$ is defined by 
\begin{equation}
\exp\left(\sum_{m=1}^\infty x_m\frac{t^m}{m!}\right)=1+\sum_{n=1}^\infty\mathbf Y_n(x_1,x_2,\dots,x_n)\frac{t^n}{n!}\,. 
\label{def:bell-p}
\end{equation}
It is expressed as  
\begin{align*}
&\mathbf Y_n(x_1,x_2,x_3,\dots,x_n)\\ 
&=\sum_{k=1}^n\sum_{i_1+2 i_2+\cdots+(n-k+1)i_{n-k+1}=n\atop i_1+i_2+i_3+\cdots=k}\frac{n!}{i_1!i_2!\cdots i_{n-k+1}!}\\
&\qquad\qquad\times \left(\frac{x_1}{1!}\right)^{i_1}\left(\frac{x_2}{2!}\right)^{i_2}\cdots\left(\frac{x_{n-k+1}}{(n-k+1)!}\right)^{i_{n-k+1}}
\end{align*} 
with $\mathbf Y_0=1$ (see, e.g., \cite[\S 3.3]{Comtet}). 
In \cite{Ko26m1} and the related papers \cite{CKm,Ko25m2,KLm5,KPm4,KWm3}, by applying the theory of Bell polynomials (some fundamental theories can be found in \cite{MacDonald95}), we explicitly expressed the $q$-multiple zeta values specifically in cases involving general indices (powers) and a small number of summands using determinants.  

By applying a similar method, it becomes possible to express Fibonomial coefficients as the determinant of a matrix that is a slightly modified form of a Hessenberg matrix.  

We need the expression of Fibonomial coefficients in terms of Bell polynomials. 

\begin{Lem}  
For integers $n$ and $k$ with $0\le k\le n$, 
$$ 
(-1)^{\binom{k}{2}}\binom{n}{k}_F 
=\frac{1}{k!}\mathbf Y_k(s_{n,1},-1!s_{n,2},2!s_{n,3}\dots,(-1)^{k-1}(k-1)!s_{n,k})\,,
$$  
where 
$$
s_{n,r}=\sum_{j=0}^{n-1}(\alpha^{n-1-j}\beta^j)^r=\alpha^{r(n-1)}\dfrac{1-(\beta/\alpha)^{r n}}{1-(\beta/\alpha)^{r}}\,.
$$
\label{lem:fibonom-bell}
\end{Lem} 
\begin{proof}  
We have 
\begin{align*}
\prod_{j=0}^{n-1}(1+X_j z)&=\exp\left(\sum_{m=1}^\infty h_m\frac{z^m}{m!}\right)\\
&=\sum_{k=0}^n\mathbf Y_k(h_1,h_2,\dots,h_k)\frac{z^k}{k!}\,,
\end{align*} 
where 
$$
\frac{h_m}{m!}=\frac{(-1)^{m-1}}{m}p_m\quad \left(p_m=\sum_{j=0}^{n-1}X_j^m\right)\,,
$$ 
that is, $h_m=(-1)^{m-1}(m-1)!\sum_{j=0}^{n-1}X_{j}^m$ ($1\le m\le k$). 
Hence, by the generating function in Lemma \ref{gf:fibonom}, we have the desired identity with $X_j=\alpha^{n-1-j}\beta^j$ and so $p_r=s_{n,r}$.  
\end{proof}

For two sequences $a_0=1,a_1,a_2,\dots$ and $b_0=1,b_1,b_2,\dots$, we have the equivalent expressions (\cite[Lemma 5]{Ko26m1}). 

\begin{Lem}  
The following expressions are equivalent.  
\begin{enumerate}
\item[(1)] $\displaystyle b_m=\frac{1}{m!}\mathbf Y_m\bigl(a_1,-1!a_2,-2!a_3,\dots,(-1)^{m-1}(m-1)!a_m\bigr)$
\item[(2)] $\displaystyle b_m=\frac{1}{m!}\left|\begin{array}{ccccc}
a_1&1&0&\cdots&0\\
a_2&a_1&2&&\vdots\\
\vdots&&\ddots&&0\\
a_{m-1}&a_{m-2}&\cdots&a_1&m-1\\
a_m&a_{m-1}&\cdots&a_2&a_1\\
\end{array}
\right|$ 
\item[(3)] $\displaystyle a_n=\left|\begin{array}{ccccc}
b_1&1&0&\cdots&0\\
2 b_2&b_1&1&&\vdots\\
\vdots&&\ddots&&0\\
(n-1)b_{n-1}&b_{n-2}&\cdots&b_1&1\\
n b_n&b_{n-1}&\cdots&b_2&b_1\\
\end{array}
\right|$ 
\item[(4)] $\displaystyle m b_m=\sum_{i=1}^m(-1)^{i-1}a_i b_{m-i}$
\item[(5)] $\displaystyle a_n=\sum_{j=1}^{n-1}(-1)^{j-1}b_j a_{n-j}+(-1)^{n+1}n b_n$ 
\end{enumerate}
\label{lem:gtrudi}
\end{Lem}

Combining Lemma \ref{lem:fibonom-bell} and Lemma \ref{lem:gtrudi} (1) and (2),  we can get a determinant expression in terms of $s_{n,r}$ ($1\le r\le n$). For convenience, we put $\delta_j=(-1)^{\binom{n}{j}}$.   
Note that only the diagonal elements immediately above the main diagonal take the values $1,2,3,\dots,n-1$, and are therefore not equal.   

\begin{theorem} 
For integers $n$ and $k$ with $0\le k\le n$,  
$$ 
\binom{n}{k}_F=\frac{1}{k!}\left|\begin{array}{ccccc}
s_{n,1}&1&0&\cdots&0\\
\delta_2 s_{n,2}&s_{n,1}&2&&\vdots\\
\vdots&&\ddots&&0\\
\delta_{k-1}s_{n,k-1}&\delta_{k-2}s_{n,k-2}&\cdots&s_{n,1}&k-1\\
\delta_k s_{n,k}&\delta_{k-1}s_{n,k-1}&\cdots&\delta_2 s_{n,2}&s_{n,1}\\
\end{array}
\right|\,.
$$  
\label{th:fibonom-bell1}
\end{theorem}

By using the inversion relation in Lemma \ref{lem:gtrudi} (2) and (3), we can obtain a determinant representation of $s_r$ in terms of Fibonomial coefficients. Note that only the first column consists of $1$ times, $2$ times, $3$ times, and so on, the respective elements.
  
\begin{theorem}  
For integers $n$ and $k$ with $0\le k\le n$, 
$$
s_{n,k}=\left|\begin{array}{ccccc}
\binom{n}{1}_F&1&0&\cdots&0\\
2\delta_2\binom{n}{2}_F&\binom{n}{1}_F&1&&\vdots\\
\vdots&&\ddots&&0\\
(k-1)\delta_{k-1}\binom{n}{k-1}_F&\delta_{k-2}\binom{n}{k-2}_F&\cdots&\binom{n}{1}_F&1\\
k\delta_r\binom{n}{k}_F&\delta_{k-1}\binom{n}{k-1}_F&\cdots&\delta_2\binom{n}{2}_F&\binom{n}{1}_F\\
\end{array}
\right|\,. 
$$ 
\label{th:bell-fibonom2}
\end{theorem}

\section{Final comments}   

One may think that the results in this paper can be generalized about a more general 
$$
\binom{n}{m}_u=\frac{u_n u_{n-1}\cdots u_1}{u_m u_{m-1}\cdots u_1\cdot u_{n-m}u_{n-m-1}\cdots u_1}\,,  
$$ 
introduced in \cite{JM49}.  However, considering the positions of $0$ in (\ref{eq:emrah10}) as well as the elements' structure in Theorem \ref{th1}, it is found that Fibonacci numbers are the best choice, so it is not easy to generalize at all.




\medskip 
\noindent  
{\bf Availability of data and materials}\\\noindent 
Not applicable.


\end{document}